\documentclass[a4paper, reqno, 11pt]{amsart}
\makeatletter

\@addtoreset{equation}{section}
\makeatother
\usepackage{setspace}
\usepackage{xcolor}
\usepackage{amsmath}
\usepackage{amssymb}
\usepackage{latexsym}
\usepackage{amsthm}
\usepackage{hyperref}
\newtheorem{Thm}{Theorem}[section]

\newtheorem{Lem}[Thm]{Lemma}

\newtheorem{Cor}[Thm]{Corollary}

\theoremstyle{definition}

\newtheorem{Rem}[Thm]{Remark}

\begin{document}

\title[]{Non-existence of measurable solutions of certain functional equations via probabilistic approaches}
\author{Kazuki Okamura}
\address{School of General Education, Shinshu University}
\curraddr{Department of Mathematics, Faculty of Science, Shizuoka University}
\email{okamura.kazuki@shizuoka.ac.jp}
\date{}

\maketitle

\begin{abstract}
This paper deals with functional equations  in the form of  $f(x) + g(y) = h(x,y)$ where $h$ is given and $f$ and $g$ are unknown. 
We will show that if $h$ is a Borel measurable function associated with characterizations of the uniform or Cauchy distributions, 
then there is no measurable solutions of the equation. 
Our proof uses a characterization of the Dirac measure and it is also applicable to the arctan equation. 
\end{abstract}

\section{Introduction}

In this paper, we consider Borel measurable solutions $f, g : B \to \mathbb{R}$ of 
\begin{equation}\label{gen} 
f(x) + g(y) = h(x,y), \  \ x, y \in B,  
\end{equation}
where $B$ is a Borel subset of $\mathbb{R}$ and  $h : B^2 \to \mathbb{R}$ is a Borel measurable function.  

Although in \eqref{gen} there are numerous choices for the Borel measurable function $h$, 
we assume that $h$ is a composition of a non-constant Borel measurable function and a Borel measurable function $H$ satisfying that 
if $X$ and $Y$ are independent random variables with a common distribution, then $H(X, Y)$ and $Y$ are also independent. 
There would be numerous choices also for such $H$, however it is not necessarily easy to find such $H$.  
The more complicated $H$ is, the harder to show the independence of $H(X, Y)$ and $Y$ would be. 

In this paper we assume that $H$ appears in the context of characterizations of the uniform distribution or the Cauchy distribution. 
Here characterization means that if $H(X, Y)$ and $Y$ are independent then $X$ follows the uniform distribution or the Cauchy distribution.  
We do {\it not} directly use those characterizations, instead we use a characterization of the Dirac measure. %

Our approach is similar to that of Smirnov \cite{Smirnov2019}. %
\cite{Smirnov2019} shows that any measurable solution of the Cauchy functional equation is locally-integrable, by applying the Kac-Bernstein theorem,  which gives a characterization of the normal distribution.   
Recently, Mania \cite{Mania2020} and Mania and Tikanadze \cite{Mania2019} consider martingale characterizations of measurable solutions of the Cauchy, Abel and Lobachevsky functional  equations.   

The purpose of this paper is showing that a probabilistic approach is  also applicable to more complicated functional equations than the Cauchy, Abel and Lobachevsky equations. 
Our proof uses a characterization of the Dirac measure, not of the normal distribution. %
 Our proof is also applicable to the arctan equation.

Now we state our main results. 
\begin{Thm}\label{main-uni}
Let $h : (0,1) \to (0,1)$ be  a Borel measurable map preserving the Lebesgue measure $\ell$. 
Let $j : (0,1] \to \mathbb{R}$ be  a Borel measurable map such that $\ell \left(\left\{x \in (0,1] : j(x) = c \right\}\right) < 1$ for every $c \in \mathbb{R}$. 
Then, there is no measurable solutions $f, g : (0,1) \to (0,1)$ of  
\begin{equation}\label{main-1} 
f(x) + g(y) = j \left(\min\left\{\frac{h (x)}{y}, \frac{1 - h (x)}{1 - y} \right\} \right), \ x, y \in (0,1). 
\end{equation} 
\end{Thm}

The conditions for the functions $h$ and $j$ are not restrictive. 
There are many candidates for measure-preserving transformations of $(0,1)$.  
For example, 
$$h(x) = \begin{cases} 2x \ \ \ \ \  \ \ 0 < x < \dfrac{1}{2} \\ \dfrac{1}{2} \ \  \ \  \ \ \ \  x = \dfrac{1}{2} \\ 2x-1 \ \ \dfrac{1}{2} < x < 1  \end{cases}$$
and
$$h(x) = \begin{cases} x + a \ \ \ \ \ \ \ \ \ 0 < x < 1-a \\ a \ \  \ \ \   \  \ \ \ \ \ \ \ \ x = 1-a  \\ x-1+a \ \ \ \ 1-a < x < 1  \end{cases}$$
are measure-preserving transformations of $(0,1)$. 

\begin{Thm}\label{main-uni-2}
Let $h : (0,1) \to (0,1)$ be  a Borel measurable map preserving the Lebesgue measure $\ell$. 
Let $j : [0,1) \to \mathbb{R}$ be  a Borel measurable map such that $\ell (\{x \in [0,1): j(x) = c\}) < 1$ for every $c \in \mathbb{R}$. 
Then, there is no Borel measurable solutions $f, g : (0,1) \to (0,1)$ of 
\begin{equation}\label{main-2} 
f(x) + g(y) = j \left(\pi(h (x) + y) \right), \ x, y \in (0,1),
\end{equation} 
where $\pi(z)$ denotes the fractional part of $z$. 
\end{Thm} 

Although \eqref{main-2} might look largely different from \eqref{main-1}, 
our proofs of these theorems are very similar to each other. 
 
Let the standard Cauchy measure be the probability measure on $\mathbb{R}$ with density $\dfrac{1}{\pi (1+x^2)}$.  
\begin{Thm}\label{main-Cauchy}
Let $h : \mathbb{R} \to \mathbb{R}$ be a Borel measurable map preserving the standard Cauchy measure. 
Let $j : \mathbb{R} \to \mathbb{R}$ be a Borel measurable map such that $\ell \left(\left\{x \in \mathbb{R}  : j(x) = c \right\}\right) < 1$ for every $c \in \mathbb{R}$. 
Then, there is no Borel measurable solutions $f, g : \mathbb{R} \to \mathbb{R}$ of  
\begin{equation}\label{main-3} 
f(x) + g(y) = j \left(\frac{h (x) + y}{1 - h (x) y}\right), \ x, y \in \mathbb{R},
\end{equation} 
\end{Thm}

We can construct measurable maps preserving the standard Cauchy measure via measurable maps preserving the Lebesgue measure on an interval. 
If $X$ follows  the standard Cauchy measure and the composition $z \mapsto f(\tan(z))$ preserves the uniform distribution on $(-\pi/2, \pi/2)$, then, 
$f(X)$ also follows the standard Cauchy measure. 

The rest of this paper is organized as follows. 
In section 2, we give a lemma for a characterization of the Dirac measure. 
In Section 3, we give proofs of Theorems \ref{main-uni}, \ref{main-uni-2} and \ref{main-Cauchy}. 
Finally in Section 4, we give a probabilistic approach to the arctan equation as a corollary of Theorem \ref{main-Cauchy}.   

\section{A lemma}

The following plays a crucial role in this paper. 
Although it is probably already known, we give a proof for the sake of completeness. %
\begin{Lem}[Characterization of Dirac measures]\label{Dirac}
If $X$ and $Y$ are independent, and, $X+Y$ and $Y$ are independent, then, $Y$ is a constant a.s. 
\end{Lem}

As we will see, this assertion  is shown by solving the complex-valued Cauchy equation. 

\begin{proof}
Throughout this proof, the symbol $i$ denotes the imaginary unit. 
Let 
$$\phi_X (t) := E[\exp(i t X)], \phi_Y (t) := E[\exp(i t Y)], \  t \in \mathbb{R}.$$
Since $X$ and $Y$ are independent, 
we have that for every $s, t \in \mathbb{R}$, 
\[ E\left[\exp\left(i \left(s (X + Y) + t Y\right) \right)\right] = E\left[\exp\left(i s (X + Y)  \right)\right] E\left[\exp\left(i  t Y \right)\right]\]
\[ = \phi_X (s) \phi_Y (s) \phi_Y (t). \]
Furthermore, 
\[ E\left[\exp\left(i \left(s (X + Y) + t Y\right) \right)\right] = E\left[\exp\left(i \left(s X + (s + t) Y\right) \right)\right]\]
\[ =  \phi_X (s) \phi_Y (s+t). \]

Since $\phi_X$ is continuous on $\mathbb{R}$ and $\phi_X (0) = 1$, 
we have that for some $\epsilon_0 > 0$, $\phi_X (s) \ne 0$ for $s \in [-\epsilon_0, \epsilon_0]$.

Therefore, 
\begin{equation}\label{limited} 
\phi_Y (s+t)  = \phi_Y (s) \phi_Y (t), \  \ \  s \in  [-\epsilon_0, \epsilon_0], \  t \in \mathbb{R}. 
\end{equation} 

Let $s \in \mathbb{R}$. 
Then, for some $n$, 
$s/n \in  [-\epsilon_0, \epsilon_0]$.
By using \eqref{limited} repeatedly, 
we have that 
\[  \phi_Y (s) = n  \phi_Y \left(\frac{s}{n}\right), \]
and for every $t \in \mathbb{R}$, 
\[ \phi_Y (s+t)  =  n  \phi_Y \left(\frac{s}{n}\right) + \phi_Y (t) =  \phi_Y (s) + \phi_Y (t).   \]

By using the fact that $\phi_Y (0) = 1$ and $\phi$ is continuous again, 
we have that for some constant $c \in \mathbb{C}$, 
\[ \phi_Y (t) = \phi_Y (1)^t = \exp(c t),\ \ t \in \mathbb{R}. \]
Let $c = a+b i$. 
Since $|\phi_Y (t)| \le 1$ for every $t \in \mathbb{R}$, we have that $a = 0$. 

By the Levy inversion formula, 
we see that for some constant $C$, 
$$P(Y = C) = 1.$$ 
\end{proof} 

\section{Proofs}

\begin{proof}[Proof of Theorem \ref{main-uni}]
We assume there exist measurable maps $f, g : (0,1) \to (0,1)$ satisfying \eqref{main-1}. 
Let $U$ and $V$ be two independent random variables with the uniform distribution on $(0,1)$.  
Then, $f(U)$ and $g(V)$ are independent.
By \cite[Theorem 3.2]{Den1992},
$\displaystyle \min\left\{\frac{h (U)}{V}, \frac{1 - h (U)}{1 - V} \right\}$ and $V$ are independent.  
By this and \eqref{main-1}, 
$f(U) + g(V)$ and $g(V)$ are independent. 
Hence by Lemma \ref{Dirac}, there exists a constant $c_0$ such that $g(V) = c_0$ a.s. 
Since $V$ is the uniform distribution on $(0,1)$, $g(y) = c_0$, a.e. $y \in (0,1)$.

For $x \in (0,1)$, let 
\[ H_{x} (y) :=  \min\left\{\frac{h (x)}{y}, \frac{1 - h (x)}{1 - y} \right\}, \ y \in (0,1). \]
Then, for each $x \in (0,1)$,  $H_x$ is continuous with respect to $y$ and 
\[ \left\{H_x (y) | y \in (0,1) \right\} = \left(\min\{h(x), 1- h(x)\}, 1\right]. \]

By the assumption, there exists a constant $c_1$ such that $\ell(\{z \in (0,1] : j(z) < c_1 \}) > 0$ and $\ell(\{z \in (0,1] : j(z) > c_1 \}) > 0$. 
Since the distribution of $h(U)$ is uniform on $(0,1)$ if  the distribution of  $U$ is uniform on $(0,1)$, 
we see that for every $\epsilon > 0$ there exists $x$ such that $h(x) < \epsilon$. 
Hence there exists $x_0 \in (0,1)$ such that $$\ell\left(\left\{z \in (\min\{h(x_0), 1- h(x_0)\},1] : j(z) < c_1 \right\}\right) > 0$$ and %
$$\ell\left(\left\{z \in (\min\{h(x_0), 1- h(x_0)\},1] : j(z) > c_1 \right\}\right) > 0.$$

By this and $\ell(\{y : g(y) \ne c_0\}) = 0$, 
we can pick $y_1$ and $y_2$ such that $g(y_1) = g(y_2) = c_0$ and $j(H_{x} (y_1)) > c_1 > j(H_{x} (y_2))$.   
However, by \eqref{main-1}, 
$$f(x) + c_0 = j(H_{x} (y_1)) = j(H_{x} (y_2)).$$
Thus we have a contradiction.  
\end{proof}

\begin{proof}[Proof of Theorem \ref{main-uni-2}]
We assume there exist measurable maps $f, g : (0,1) \to (0,1)$ satisfying \eqref{main-2}. 
Due to \cite[Theorem 3.2]{Den1992} and Lemma \ref{Dirac}, 
we can show that for some constant $c_0$, $g(y) = c_0$, a.e. $y \in (0,1)$, in the same manner as in the proof of Theorem \ref{main-uni}. 

For $x \in (0,1)$, let 
\[ H_{x} (y) :=  \pi(h(x) + y), \ y \in (0,1). \]
Then, for each $x \in (0,1)$,  
\[ \left\{H_x (y) | y \in (0,1) \right\} = [0,1) \setminus \{h(x)\}. \]

By the assumption, there exists $c_1$ such that $\ell(\{z \in [0,1) : j(z) < c_1 \}) > 0$ and $\ell(\{z \in [0,1) : j(z) > c_1 \}) > 0$. 
By this and $\ell(\{y : g(y) \ne c_0\}) = 0$, 
we can pick $y_1$ and $y_2$ such that $g(y_1) = g(y_2) = c_0$ and $j(H_{x} (y_1)) > c_1 > j(H_{x} (y_2))$.   
However, by \eqref{main-2}, 
$$f(x) + c_0 = j(H_{x} (y_1)) = j(H_{x} (y_2)).$$
Thus we have a contradiction.  
\end{proof} 

\begin{proof}[Proof of Theorem \ref{main-Cauchy}]
We assume there exist measurable maps $f, g : \mathbb{R} \to \mathbb{R}$ satisfying \eqref{main-3}. 

Let $X, Y$ be two independent standard Cauchy distributions. 
Then, by the assumption, $h(X)$ and $Y$ are  independent standard Cauchy distributions.
By \cite{Arnold1979}, $\dfrac{h(X)+Y}{1-XY}$ and $Y$ are independent. 
Hence, $j \left( \dfrac{h(X)+Y}{1- h(X) Y} \right)$ and $g\left( Y \right)$ are independent. 
By Lemma \ref{Dirac} and \eqref{main-3},  
we have that for some constant $c_0$, $g(y) = c_0$, a.e. $y \in \mathbb{R}$.

For $x \in \mathbb{R}$, let 
\[ H_{x} (y) :=  \frac{h (x) + y}{1 - h (x) y}, \  \ y \in \mathbb{R} \setminus \{h(x)\}. \]
Then,   
\[ \left\{H_x (y) | y \in \mathbb{R}  \setminus \{h(x)\} \right\} = \begin{cases} \mathbb{R} \setminus \{-\frac{1}{h(x)}\} \   \ h(x) \ne 0  \\ \mathbb{R} \ \ \ \ \ \  h(x) = 0 \end{cases}. \]

By the assumption, there exists $c_1$ such that $\ell(\{z \in \mathbb{R} : j(z) < c_1 \}) > 0$ and $\ell(\{z \in \mathbb{R}  : j(z) > c_1 \}) > 0$. 
By this and $\ell(\{y : g(y) \ne c_0\}) = 0$, 
we can pick $y_1$ and $y_2$ such that $g(y_1) = g(y_2) = c_0$ and $j(H_{x} (y_1)) > c_1 > j(H_{x} (y_2))$.   
However, by \eqref{main-3}, 
$$f(x) + c_0 = j(H_{x} (y_1)) = j(H_{x} (y_2)).$$
Thus we have a contradiction.  
\end{proof}

\begin{Rem}
(i) In general, if $Z$ is a real-valued random variable and $f : \mathbb{R} \to \mathbb{R}$ is Borel measurable, then, $f(X)$ is also a real-valued random variable. 
However if $f : \mathbb{R} \to \mathbb{R}$ is {\it Lebesgue measurable}, $f(X)$ may not be a real-valued random variable. 
We give an example. 
Let $\varphi : [0,1] \to [0,1]$ be the Cantor function. 
Indeed, let $X(w) := \inf\{y : \varphi(y) = w \}$, then we can regard this as a random variable on $([0,1], \overline{\mathcal{B}([0,1])}, \ell)$, where $\overline{\mathcal{B}([0,1])}$ is the completion of the Borel $\sigma$-algebra with respect to the Lebesgue measure.  
Let $A$ be a non Lebesgue measurable subset of $[0,1]$. 
Let $f(x) = \begin{cases} 1  \ \ x \in X(A)  \\ 0 \ \ x \notin X(A) \end{cases}$. 
Then, 
$X(A)$ is contained in the Cantor set and hence the measure of $X(A)$ is zero and in particular it is Lebesgue measurable.  
Hence $f$ is Lebesgue measurable. 
However $f(X)$ is not Lebesgue measurable.  
This is a minor thing and it does {\it not} invalidate \cite[Section 2]{Smirnov2019}, because there exists a Borel measurable function $g  : \mathbb{R} \to \mathbb{R}$ such that $f = g$, a.e. 
If $g$ is locally-integrable, then $f$ is also  locally-integrable. \\
(ii) We can establish  Theorems \ref{main-uni}, \ref{main-uni-2} and \ref{main-Cauchy} for the case that $j$ is Lebesgue measurable %
and  $h$ is a measure-preserving map  on the completed measure space.  
Since the Cauchy measure is equivalent to the Lebesgue measure, in the statements of Theorems \ref{main-uni}, \ref{main-uni-2} and \ref{main-Cauchy},  we can replace the Borel sigma-algebra with its completion with respect to the Lebesgue measure.  %
\end{Rem}

\section{A probabilistic approach to the arctan equation}

We can even show that every measurable solution of the arctan equation is zero by using Theorem \ref{main-Cauchy}. %
More specifically, 

\begin{Cor}\label{main-arctan}
Every Borel measurable function $f : \mathbb{R} \to \mathbb{R}$ satisfying 
\begin{equation}\label{arctan}
f(x) + f(y) = f\left(\frac{x+y}{1-xy}\right), \ xy \ne 1. 
\end{equation}
is limited to the function $f(x) = 0$ for every $x \in \mathbb{R}$.  %
\end{Cor}

\eqref{arctan} is called the arctan equation. 
Kiesewetter \cite{Kiesewetter1965} showed that every continuous solution of \eqref{arctan} is the constant function taking zero at every point. 
Crstici, Muntean and Vornicescu \cite{Crstici1983} consider \eqref{arctan} on $\{(x,y) | xy < 1\}$ under some additional assumptions
Losonczi \cite{Losonczi1990} shows that \eqref{arctan} has a form of $A (\arctan x)$, where $A$ is an additive function on $\mathbb{R}$ with period $\pi$.

\begin{proof}[Proof of Corollary \ref{main-arctan}]
Assume that $f$ is a Borel measurable solution of \eqref{arctan}. 
By applying Theorem \ref{main-Cauchy} to the case that $g(x) = f(x)$, $h(x) = x$, $j(x) = f(x)$, 
we see that there is no other cases than the case that there exists a constant $C$ such that  
$f(y) = C$ almost everywhere with respect to the Lebesgue measure. 

Let $x_0 \in (0,1)$ be a real number such that $f(x_0) = C$. 
Let $F_0 (y) := \dfrac{x_0 + y}{1 - x_0 y}$.  
Then it is Lipschitz continuous on $\left[-x_0 / 2, x_0 / 2 \right]$. 
Hence, by Rudin \cite[Lemma 7.25]{Rudin1987}, 
\[ \ell \left(\left\{F_0 (y): y \in \left[- \frac{x_0}{2},  \frac{x_0}{2} \right], f(y) \ne C  \right\} \right) = 0, \]
where $\ell$ denotes the one-dimensional Lebesgue measure. 
Since $x_0 \in (0,1)$, 
we have that 
\[ F_0^{\prime}(y)  \ge 4 x_0 (1+x_0) \left(1 - \frac{x_0}{2} \right) > 0, \ y \in \left[ - \frac{x_0}{2},  \frac{x_0}{2} \right],  \]
and hence, 
\[ \ell \left(\left\{F_0 (y): y \in \left[  - \frac{x_0}{2},  \frac{x_0}{2} \right] \right\} \right) > 0. \]
Hence there exists $y_0$ such that $$f(y_0) = f(F_0 (y_0)) = C.$$
Therefore we have that  $$2C = f(x_0) + f(y_0) = f(F_0 (y_0)) = C,$$
which implies $C = 0$. 

We finally show that $f(x) = 0$ for every $x \in \mathbb{R}$. 
Let $x \ne 0$ and $x \ne -1$. 
Let $F (y) := \dfrac{x + y}{1 - x y}$.  
Then this function is well-defined and 
\[ F^{\prime}(y)  = \frac{x + x^2 - x^2 y + x y}{(1 - x y)^2} \ne 0 \]
on a neighborhood $[-\epsilon, \epsilon]$ of $0$. 
In particular $F$ is strictly monotone on  $[-\epsilon, \epsilon]$, and hence, the inverse $F^{-1}$ is Lipschitz on $F([-\epsilon, \epsilon])$. 

Then, by using Rudin \cite[Lemma 7.25]{Rudin1987} again, 
\[ \ell\left( \left\{y \in [-\epsilon, \epsilon] : f(F(y)) \ne 0 \right\}\right) = \ell\left( \left\{F^{-1}(z) : z \in F([-\epsilon, \epsilon]),  f(z) \ne 0 \right\}\right) = 0. \]

Therefore we have that 
\[ f(x) = f(x) + \frac{1}{2\epsilon} \int_{-\epsilon}^{\epsilon} f(y) dy = \frac{1}{2\epsilon} \int_{-\epsilon}^{\epsilon} f(F(y)) dy = 0. \]

By \eqref{arctan}, we have that $f(0) + f(y) = f(y)$ and hence $f(0) = 0$. 
By \eqref{arctan} again, we have that 
$$f(-1) + f(y) = f\left(\frac{y-1}{y+1}\right), \ \ y \ne -1. $$
By substituting $y=1$ in this equation, we see that $f(-1) = 0$.  %

Thus we have that $f(x) = 0$ for every $x \in \mathbb{R}$.  
This completes the proof of Corollary \ref{main-arctan}.
\end{proof}

\bibliographystyle{amsplain}
\bibliography{compare-fe}

\providecommand{\bysame}{\leavevmode\hbox to3em{\hrulefill}\thinspace}
\providecommand{\MR}{\relax\ifhmode\unskip\space\fi MR }
\providecommand{\MRhref}[2]{%
  \href{http://www.ams.org/mathscinet-getitem?mr=#1}{#2}
}
\providecommand{\href}[2]{#2}
\begin{thebibliography}{1}

\bibitem{Arnold1979}
Barry~C. Arnold, \emph{Some characterizations of the {C}auchy distribution},
  Australian Journal of Statistics \textbf{21} (1979), 166--169.

\bibitem{Crstici1983}
Borislav Crstici, Ioan Muntean, and Neculae Vornicescu, \emph{General solution
  of the arctangent functional equation}, L'Analyse Num\'{e}rique et de
  Th\'{e}orie de l'Approximation \textbf{12} (1983), 113--123.

\bibitem{Den1992}
Lih-Yuan Dengand and E.~Olusegun George, \emph{Some characterizations of the
  uniform distribution with applications to random number generation}, Annals
  of Institute of Statistical Mathematics \textbf{44} (1992), 379--385.

\bibitem{Kiesewetter1965}
Helmut Kiesewetter, \emph{Uber die arc tan-{F}unktionalgleichung, ihre
  mehrdeutigen, stetigen {L}osungen und eine nichtstetige {G}ruppe},
  Friedrich-{S}chiller-{U}niversitat {J}ena. {W}issenschaftliche {Z}eitschrift.
  {N}aturwissenschaftliche {R}eihe \textbf{14} (1965), 417--421.

\bibitem{Losonczi1990}
Laszlo Losonczi, \emph{Local solutions of functional equations}, Dru\v{s}stvo
  Matemati\v{c}ara i Fizi\v{c}ara S. R. Hrvatske. Glasnik Matemati\v{c}ki.
  Serija III \textbf{25(45)} (1990), no.~1, 57--67.

\bibitem{Mania2020}
Michael {Mania}, \emph{{A probabilistic method of solving Lobachevsky's
  functional equation}},  \textbf{95} (2021), no.~2, 237--243.

\bibitem{Mania2019}
Michael Mania and Luca Tikanadze, \emph{Functional equations and martingales},
  preprint, available at arXiv:1912.06299v2.

\bibitem{Rudin1987}
Walter Rudin, \emph{Real and complex analysis}, 3rd ed., Mc{G}raw-{H}ill, 1987.

\bibitem{Smirnov2019}
Sergey~N. Smirnov, \emph{A probabilistic note on the {C}auchy functional
  equation}, Aequationes Mathematicae \textbf{93} (2019), 445--449.

\end{thebibliography}

\end{document}